\documentclass[12pt]{article}
\usepackage{latexsym,amssymb,amsmath,amsfonts,amsthm}
\newtheorem{theorem}{Theorem}

\newtheorem{lemma}{Lemma}

\def\beq{ \begin{equation} }
\def\eeq{ \end{equation} }
\def\mn{\medskip\noindent}

\def\ep{\epsilon}

\def\square{\vcenter{\vbox{\hrule height .4pt
  \hbox{\vrule width .4pt height 5pt \kern 5pt
        \vrule width .4pt} \hrule height .4pt}}}

\def\TT{\mathbb{T}}
\def\ZZ{\mathbb{Z}}

\def\clearp{}

\begin{document}

\title{The Contact Process on Periodic Trees }
\author{Yufeng Jiang, Remy Kassem, Grayson York, Brandon Zhao, \\
Xiangying Huang, Matthew Junge, and Rick Durrett \\
\small Dept of Math, Duke University, Box 90320, Durham NC 27708-0320}

\date{\today}		

\maketitle

\begin{abstract}
A little over 25 years ago Pemantle pioneered the study of the contact process on trees, and showed that on homogeneous trees the critical values $\lambda_1$ and $\lambda_2$ for global and local survival were different. He also considered trees with periodic degree sequences, and Galton-Watson trees. Here, we will continue the study of the periodic case. Two significant new results give sharp asymptotics for the critical value $\lambda_2$ of $(1,n)$ trees, $\lambda_2(n) \sim \sqrt{0.5( \log n)/n}$, and generalize that result to the $(a_1,\ldots, a_k, n)$ tree when $\max_i a_i \le n^{1-\ep}$ and $a_1 \cdots a_k = n^b$. We also give results for $(a,b,c)$ trees. Our results in this case improve those found by Pemantle. However, the values come from solving cubic equations, so the explicit formulas are not pretty, but it is surprising that they depend only on $a+b+c$ and $abc$.
\end{abstract}

\section{Introduction}

The contact process can be defined on any graph as follows: occupied sites become vacant at rate 1, while
vacant sites become occupied at rate $\lambda$ times the number of occupied neighbors. 
Harris (1984) introduced the contact process on $\ZZ^d$ where it has been extensively studied. See Liggett (1999) for a summary of
most of what is known. 

Pemantle (1992) began the study of contact processes on trees.
Let $\xi_t$ be the set of occupied sites at time $t$ and use $\xi^0_t$ to denote the process with $\xi^0_0 = \{0\}$
where 0 is the root of the tree. His main new result was that the process had two phase transitions:
\begin{align*}
\lambda_1 &= \inf\{ \lambda : P( \xi^0_t \neq \emptyset \hbox{ for all $t$}) > 0 \} \\
\lambda_2 &=\inf\{ \lambda : P( 0 \in \xi^0_t \hbox{ infinitely often}) > 0 \}
\end{align*}
Let $\TT_d$ be the tree in which each vertex has $d+1$ neighbors. Pemantle showed that $\lambda_1< \lambda_2$ when $d\ge 3$
by getting upper bounds on $\lambda_1$ and lower bounds on $\lambda_2$. Liggett (1996) proved that in $d=2$
$\lambda_1 < 0.605 < 0.609 < \lambda_2$ to settle the last case. $\TT_1=\ZZ$, which has $\lambda_1=\lambda_2$. 
In (1996) Stacey gave an elegant proof that on $\TT_d$ and a number of other graphs we have $\lambda_1< \lambda_2$.  

To complement the results for the contact process,
we will also consider branching random walk $\zeta_t$ in which $\zeta_t(x)$ is the number of particles at $x$ at time $t$.
Intuitively, it is the contact process without the restriction of one particle per site. Particles give birth onto
neighboring sites at rate $\lambda$ and die at rate 1. Let $\lambda_g$ and $\lambda_\ell$ be the critical values for
branching random walk that correspond to $\lambda_1$ and $\lambda_2$. Pemantle and Stacey (2001) considered the contact process
and branching random walk on Galton-Watson trees in which each vertex has an independent and identically distributed number of children.
They found situations in which $\lambda_1=\lambda_2$ or $\lambda_g = \lambda_\ell$.

Here, we will concentrate on periodic trees. Although simpler than Galton-Watson trees, we are able to prove very detailed results which quantify the observation that the critical value is determined by the largest degree, and have unexpected symmetries in the period three care. To define the class of periodic trees and to study their properties,  
it is convenient to define a integer valued function $\ell(x)$ on the tree which we think of as the height of $x$. The children of $x$,
which we will denote by $x+$ have $\ell(x+) = \ell(x)+1$ and its parent, denoted by $x-$, has value $\ell(x-) =\ell(x)-1$.  The degree 
of $x$ is $g(\ell(x))$ where $g$ is a periodic function on $\ZZ$. If the period is two, and the values of $g$ are $a$ and $b$ we call
the result the $(a,b)$ tree.

\begin{figure}[ht]
\begin{center}
\begin{picture}(200,200)
\put(130,160){\line(4,1){40}}
\put(130,160){\line(4,-1){40}}
\put(130,120){\line(4,1){40}}
\put(130,120){\line(4,-1){40}}
\put(130,80){\line(4,1){40}}
\put(130,80){\line(4,-1){40}}
\put(130,40){\line(4,1){40}}
\put(130,40){\line(4,-1){40}}
\put(90,140){\line(2,1){40}}
\put(90,140){\line(2,-1){40}}
\put(90,60){\line(2,1){40}}
\put(90,60){\line(2,-1){40}}
\put(50,100){\line(1,1){40}}
\put(50,100){\line(1,-1){40}}
\put(50,100){\line(-1,2){35}}
\put(88,65){0}
\put(5,10){$\ell(x)=-1$}
\put(88,10){0}
\put(128,10){1}
\put(168,10){2}
\put(85,145){$x-$}
\put(128,165){$x$}
\put(172,168){$x+$}
\put(172,148){$x+$}
\end{picture}
\caption{Picture of the height function $\ell(x)$.}
\end{center}
\label{fig:ell}
\end{figure}
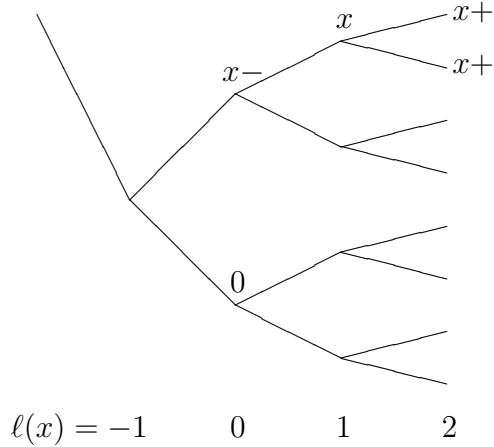

\subsection{Elementary results}

Here the results, with the possible exception of Theorem \ref{lamell}, are not new. However, the proofs are short and set
the stage for our new results. 

\begin{theorem} \label{l1bds}
On the $(a,b)$ tree
$$
\frac{1}{\sqrt{(a+1)(b+1)}} =\lambda_g \le \lambda_1 \le \frac{1}{\sqrt{ab}-1}
$$
\end{theorem}

\noindent
As $a,b \to\infty$ the upper and lower bounds are asymptotic to $1/\sqrt{ab}$, which gives the growth rate of the tree of descendants of a fixed vertex. 
When $a=3$ and $b=4$ the bounds are
$$
0.2236 = 1/\sqrt{20}  =\lambda_g \le \lambda_1 \le 1/(\sqrt{12}-1) = 0.4058
$$

To compute $\lambda_\ell$ we begin with a general result. If $M_0(v,2n)$ is the number of paths of length $2n$ that begin and end at $v$ then
\beq
M_0(v,2m) \cdot M_0(v,2n) \le M_0(v,2(m+n))
\label{subm}
\eeq
From this it follows easily that

\begin{lemma} \label{subadd}
As $n\to\infty$
$$
M_0(v,2n)^{1/2n} \to \sup_{k\ge 1} M_0(v,2k)^{1/2k} \equiv M_0
$$
\end{lemma}

\noindent
If the graph is connected then the limit is independent of $v$. 

\begin{theorem} \label{brwcv}
The critical value for local survival of the branching random walk is 
$$
\lambda_\ell = 1/M_0.
$$
\end{theorem}

\noindent
This result is Lemma 3.1 in Stacey and Pemantle (2001). To apply this result we have to compute $M_0$.

\begin{theorem} \label{lamell}
On the $(a,b)$ tree
$$
\lambda_\ell = \frac{ 1}{ \sqrt{ a} + \sqrt{b} }
$$
\end{theorem} 

\noindent 
Note that when $a=b=d$ this is $1/2\sqrt{d}$. When $a=3$, $b=4$ the answer is 
$$
1/(\sqrt{3} + \sqrt{4}) = 0.267949.
$$

Since $\lambda_2 \ge \lambda_\ell$ the next result is an immediate corollary.

\begin{theorem} \label{ablam2}
On the $(a,b)$ tree  
$$
 \frac{1}{\sqrt{a} + \sqrt{b}} \le \lambda_2 
$$
\end{theorem}

\noindent
A little arithmetic shows
$$
\frac{1}{\sqrt{a} + \sqrt{b}} - \frac{1}{\sqrt{ab}-1} = (\sqrt{a}-1)(\sqrt{b}-1) - 2
$$
This allows us to conclude $\lambda_1 < \lambda_2$ when $a=b=6$ or $a=5$, $b=7$. Fortunately, Stacey's (1996)
result allows us to conclude $\lambda_1 < \lambda_2$ when $ab>1$. See Example 2 on page 1720. 

Though we already have one proof of Theorem \ref{ablam2}, we will give another, because the new proof will extend to trees with period 
three. To motivate the new proof, we recall the proof of the analogous lower bound $\lambda_2 > 1/2\sqrt{d}$ for $\TT_d$. 
Using the height $\ell(x)$ introduced earlier, define a weight function 
$$
W(\xi_t) = \prod_{x\in \xi_t} \rho^{\ell(x)}
$$
If we take $\rho=1/\sqrt{d}$ then when $\lambda< 1/2\sqrt{d}$ we have $EW(\xi_t) \le -\ep EW(\xi_t)$ for some $\ep>0$. From this we see
that $Ew(\xi^0_t) \le e^{-\ep t}$ and the expected number of
visits to 0 is finite. To generalize to the $(a,b)$ tree we show that 

\begin{lemma} \label{pickab}
If $\lambda < 1/(\sqrt{a}+\sqrt{b})$ then we can pick values $g(a),g(b) >0$ so that if $h(x+) = h(x) g(d(x))$ then 
$$
\bar W(\xi_t) = \prod_{x\in \xi_t} h(x)
$$
is harmonic, i.e., $(d/dt) E W(\xi_t)=0$. 
\end{lemma}

\noindent
See Section \ref{sec:PfTh4}. After our proof was written we discovered that Pemantle \cite{Pem92}, see the text above (8) on his page 2103, had a much simpler proof. He set $g(a) = 1/\sqrt{b}$ and $g(b) = 1/\sqrt{a}$ and checked that the resulting function was superharmonic 
when $\lambda < 1/(\sqrt{a}+\sqrt{b})$. For more details see the end of our Section \ref{sec:PfTh23}.

\subsection{New results for $\lambda_2$}

Pemantle also showed, see the text below (7) on the same page,
that for a general period $k$ tree $\lambda_2 \le C/(a_1 \cdots a_k)^{1/2k}$. When $k=2$ the bound is $C/(ab)^{1/4}$.
When $a=b$.
$$
1/(\sqrt{a} + \sqrt{b}) = 1/2\sqrt{a} \qquad C/(ab)^{1/4} = C/\sqrt{a}
$$
so the bounds differ by a factor of 2. However when $a=1$ and $b=n$ 
\beq
1/(\sqrt{a} + \sqrt{b}) = 1/(1+\sqrt{n}) \qquad C/(ab)^{1/4} = C/n^{1/4}
\label{twpbds} 
\eeq
so they are different orders of magnitude.

To see which bound is closer to the truth, we will investigate this special case. 

\begin{theorem} \label{abub2}
On the $(1,n)$ tree, as $n \to \infty$ the critical value 
$$
\lambda_2 \sim  \sqrt{c(\log n)/n}\quad\hbox{where $c=1/2$}.
$$
\end{theorem}

\noindent
In contrast Pemantle's best result for the constant degree $n$ tree is
$$
\frac{4 - \sqrt{8}}{2} \le \liminf_{n\to\infty} \sqrt{n} \lambda_2  
 \le \limsup_{n\to\infty} \sqrt{n} \lambda_2 \le e
$$

In historical order, the steps in the proof of Theorem \ref{abub2} are as follows. We first prove a lower bound on the persistence of the contact process on a star graph. We use the methods in Section 2 in Chatterjee and Durrett (2009) but our new reuslt is asymptotically sharp when $\sqrt{n}\lambda \to\infty$. Following that reference, we will refer to vertices of degree $n$ as stars. We next obtain lower bounds on the probability of transferring infection from one star to a neighboring star (at distance 2 on the tree), and then compare with 1-dependent oriented percolation on $\ZZ$ to prove an upper bound with $c=1+\ep$. Once this was done we found a remarkably simple upper bound on the survival time of stars that allowed us to prove a lower bound on the critical value with $c=1/2-\ep$. To close the gap between the two bounds, we replaced our comparison with the contact process on $\ZZ$ by one that uses all of the stars on the $(1,n)$ tree. There are many block constructions that compare with oriented percolation on $\ZZ$, but this is the first that know of where the comparison is with a processon a tree.

The last proof extends easily to the $(1, \ldots 1,n)$ tree. If there are $k$ 1's then the constant $c=1/2$ is replaced by $c=k/2$. 
Pemantle \cite{Pem92} has an upper bound on $\lambda_2$ that covers this case and trees with longer periods. 
To describe the application of his result to the $(1, \ldots 1,n)$ tree it is useful to draw a picture. See Figure 2.
In this example there is a path of length $j_2=1$ from $\sigma$ to each of the $v_i$ that is disjoint from the path of length 
$j_1=4$ from $\rho$ to $\sigma$. To build the infinite tree we repeatedly attach copies of the original graph to the vertices
$v_1, \ldots v_m$. The result in this case is a $(1,1,1,1,n)$ tree. Changing Pemantle's notation to match ours.

\begin{figure}[ht]
\begin{center}
\begin{picture}(160,160)
\put(30,80){\line(1,0){100}}
\put(110,80){\line(1,1){20}}
\put(110,80){\line(1,2){20}}
\put(110,80){\line(1,3){20}}
\put(110,80){\line(2,1){20}}
\put(110,80){\line(2,3){20}}
\put(110,80){\line(2,5){20}}
\put(110,80){\line(1,-1){20}}
\put(110,80){\line(1,-2){20}}
\put(110,80){\line(1,-3){20}}
\put(110,80){\line(2,-1){20}}
\put(110,80){\line(2,-3){20}}
\put(110,80){\line(2,-5){20}}
\put(28,77){$\bullet$}
\put(48,77){$\bullet$}
\put(68,77){$\bullet$}
\put(88,77){$\bullet$}
\put(108,77){$\bullet$}
\put(30,90){$\rho$}
\put(105,90){$\sigma$}
\put(135,140){$v_1$}
\put(138,75){$\vdots$}
\put(135,16){$v_n$}
\end{picture}
\caption{$(1,1,1,1,n)$ tree in Pemantle's notation.}
\end{center}
\label{fig:11n}
\end{figure}
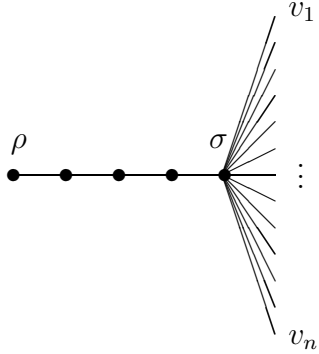

\begin{theorem} 
Let $j=j_1+j_2$ and let $r=\max\{2, \lceil j/\ln(n) \rceil\}$. There is a constant $c_4$ so that
$$
\lambda_2 \le c_4 \sqrt{r \ln(r) \ln(n) /n }
$$
\end{theorem}

\noindent
As Pemantle says on page 2103, ``For ease of exposition, large constants are chosen with reckless abandon.'' It is not easy to 
trace through the proof to get a value of $c_4$ but it is in the hundreds if not thousands. In the example drawn $j=5$ and $n=13$,
so $r=2$. 

Our next result generalizes our conclusion for the $(1, \ldots, 1,n)$ and gives a result that is asymptotically sharp.
Given the proof of Theorem \ref{abub2} very little work is required.

\begin{theorem} \label{akub2}
Consider the $(a_1,a_2, \ldots, a_k, n)$ tree with $A = \max_i a_i \le n^{1-\ep}$ and $a_1 a_2 \cdots a_k = n^b$. 
As $n \to \infty$ the critical value $\lambda_2 \sim  \sqrt{c(\log n)/n}$ where 
$$
c = \frac{k-b}{2}.
$$
\end{theorem}

\noindent
Note that within this class the critical value depends only on the maximum degree, as Pemantle says at the bottom of page 2103.
When $b=0$ and $k=1$ this reduces to the result for the $(1,n)$ tree. Since $A \le n^{1-\ep}$, $b<k$ and the constant is positive.
Due to the constraint $A \le n^{1-\ep}$ we cannot take $b=k$ in the theorem. If we did then we would have the $k$-tree which 
has $\lambda_2 \le C/\sqrt{k}$ so $c=0$.

\subsection{Period three trees}

Suppose the degrees are $a,b,c$. Our calculations will lead to cubic equations, which make the formulas
ugly, but it is remarkable that the answers depend only on $a+b+c$ and $abc$. To solve the cubic and get an explicit formula, we can use
Cardano's formula. A solution of $x^3 + px + q =0$ is given by 
\beq
\left( -\frac{q}{2} + \sqrt{ \frac{q^2}{4} + \frac{p^3}{27} } \right)^{1/3} 
 +\left( -\frac{q}{2} - \sqrt{ \frac{q^2}{4} + \frac{p^3}{27} } \right)^{1/3} 
\label{cardano}\\
\eeq

Starting with the first critical value, $\lambda_g = 1/\Lambda$ where $\Lambda$ is the max eigenvalue of 
$$
A = \begin{pmatrix} 0 & a & 1 \\ 1 & 0 & b \\ c & 1 & 0 \end{pmatrix}
$$
To see this note that $A^n_{ij}$ gives the number of paths of length $n$ from $i$ to $j$ and by the Perron-Frobenius theorem
implies that $A^n_{ij} \sim r_i \Lambda^n \ell_i$ where $\ell$ and $r$ are left and right eigenvectors associated to $\Lambda$. 
The eigenvalues are roots of $x^3 -(a+b+c)x -(abc+1)=0$, so the maximum eigenvalue can be found using \eqref{cardano}. 
As explained in Section \ref{sec:per3} we know this is the correct root since the maximum eigenvalue is simple and 
in some cases the other two roots are complex.

\begin{theorem}\label{lam13}
$$
\frac{1}{\Lambda} = \lambda_g \le \lambda_1 \le \frac{1}{(abc)^{1/3} - 1 }
$$
\end{theorem}

\noindent
The proof of the second result is the same as the proof of 
the second result in Theorem \ref{l1bds}. 
As in the case of $(a,b)$ trees considered in Theorem \ref{l1bds} the bounds are not very good
$$
\begin{matrix}
a,b,c & \lambda_g \ge & \lambda_1 \le \\
2,3,4 & 0.1711 & 0.5306 \\
3,4,5 & 0.1270 & 0.3430 \\
4,6,8 & 0.0865 & 0.2097 \\
6,8,10 & 0.0638 & 0.1464
\end{matrix}
$$

To find a lower bound on $\lambda_\ell \le \lambda_2$ we will generalize Lemma \ref{pickab} to period three trees.
To do this we need to find the solution of 
\begin{align*}
& a g(b) + 1/g(a) = 1/\lambda \\
& b g(c) + 1/g(b) = 1/\lambda \\
& c g(a) + 1/g(c) = 1/\lambda 
\end{align*}
Our calculations, see Section \ref{sec:per3}, show that if $\lambda < \lambda_0$ the smallest positive root of 
$$
\frac{1}{\lambda^6} - \frac{2(a+b+c)}{\lambda^4} + \frac{(a+b+c)^2}{\lambda^2} - 4abc  = 0 
$$
then the solutions $g(a), g(b), g(c)$ are positive, so we can define a harmonic weight function. The proof of this is just calculus, but we would be surprised if the reader can find it without looking at our solution.

Let $x = 1/\lambda^2$ to convert the last cubic into
$$
x^3 - 2(a+b+c)x^2 + (a+b+c)^2x - 4abc  = 0 
$$
Using some elementary analysis, we shows that this cubic has three real roots and that if $x > x_0(a,b,c)$, the largest root, then $g(a), g(b), g(c)>0$.
It follows that

\begin{theorem} \label{lam23}
$$
\frac{1}{x_0(a,b,c)^{1/2}} \le \lambda_\ell \le \lambda_2
$$
\end{theorem}

Cardano's formula can be used to find a formula for $x_0(a,b,c)$. However for concrete examples
it is easier to just solve the cubic numerically. To illustrate Theorem \ref{lam23}, consider 
$$
\begin{matrix}
a,b,c & x_0(a,b,c) & \hbox{lower bound} \\
2,3,4 & 11.847 & 0.2905 \\
3,4,5 & 15.887 & 0.2509 \\
4,6,8 & 23.693 & 0.2054 \\
6,8,10 & 31.774 & 0.1774
\end{matrix}
$$ 
Since Pemantle's result contains an unspecified constant, we cannot compare with his result.
Only in the last case is the lower bound on $\lambda_2$ larger than the upper bound on $\lambda_1$.

The remainder of the proof is devoted to proofs. Theorem 1 is proved in Section 2, Theorems 2 and 3 in Section 3, Theorem 4 in Section 4,
Theorems 5 and 8 in Section 5, Theorems 9 and 10 in Section 6.

\clearp

\section{Proof of Theorem \ref{l1bds}} \label{sec:PfTh1}

\begin{proof} 
The number of paths of length $2n$ on the tree is $(a+1)^n(b+1)^n$. The second ingredient is

\begin{lemma} \label{Etrails}
The expected number of branching random walk infection paths of length $n$ corresponding to a fixed path is $\lambda^n$.
\end{lemma}

\begin{proof}
Let $x_0, \ldots x_n$ be the sequence of sites. If we condition on the lifetime of 
the individual at $x_i$ we see that the mean number of infections from 
$x_i$ to $x_{i+1}$ is
$$
\int_0^\infty e^{-s} \lambda \, ds = \lambda
$$
Since the infection events are independent the desired result follows.
\end{proof}

\noindent
If $\lambda < \sqrt{(a+1)(b+1)}$ then the expected number of branching random walk infection paths of even length is
$$
\sum_{m=1}^{\infty} (a+1)^m (b+1)^m \lambda ^{2m} < \infty
$$
so the process dies out globally.

To prove the second result, we only allow infections that go away from the root. Each vertex will be occupied at most once,
so the number of infected sites at time $t$ is branching process. 
The probability an individual infects a neighbor before they become healthy
is $\lambda/(\lambda+1)$, so if 
$$
\frac{a\lambda}{1+\lambda} \cdot \frac{b\lambda}{1+\lambda}  > 1
$$ 
the branching process of infections on the oriented tree is supercritical, and infection survives globally. 
Rearranging the formula, the condition becomes $\lambda > 1/(\sqrt{ab}-1)$. 
\end{proof}

\clearp

\section{Proof of Theorems \ref{brwcv} and  \ref{lamell}} \label{sec:PfTh23}

\begin{proof} [Proof of Theorem \ref{brwcv}.] It follows from Lemma \ref{subadd} that $M_0(0,2n) \le M_0^{2n}$. If $\lambda < 1/M_0$
then the expected number of infection trails of any length is
$$
\sum_{n=1}^\infty \lambda^{2n}M_0(0,2n) < \infty
$$
so $\lambda < \lambda_\ell$.

Suppose $\lambda > 1/M_0$. Due to Lemma \ref{subadd} if $n$ is large enough $M_0(0,2n) \lambda^{2n} > 1$.
Let $X_n$ be the number of infection trails of length $2n$ from 0 to 0. $X_{kn}$ dominates a supercritical
branching process, so with positive probability the root will be visited infinitely many times. 
\end{proof} 

\mn
{\bf Computation of $M_0$.} We begin with the homogeneous case in which $a=b=d$. In order for a path that starts at the root 0 to end 
at a vertex with $\ell(x)=0$ there needs to be an equal number of up an down steps. There are $d$ choices when we go up but
only 1 when we go down, so the number of paths of length $2n$ starting from 0 to ending at a vertex with $\ell(x)=0$ is
$$
\binom{2n}{n} d^n
$$
A path $x_0=0, x_1, x_2, \ldots x_{2n}$ with $\ell(x_{2n})=0$ that has $\ell(x_i) \ge 0$ is guaranteed to end at 0. Call these good paths.
If we have an arbitrary path $y_0=0, y_1, y_2, \ldots y_{2n}$ with $\ell(y_{2n})=0$ rewrite is as a sequence $w_1, \ldots w_{2n}$ of $n$ $U$'s for up 
and $n$ $D$'s for down. Let $v = \min\{ \ell(y_i) : 0 \le i \le 2n \}$ and $k = \min\{ j : \ell(y_j) = v \}$. Consider a new path with $U,D$ sequence 
$w_{k+1}, \ldots w_{2n}, w_1 \ldots w_k$. If $v=0$, $k=0$ and the path does not change. However if $v<0$ this is a new path 
 $z_0=0, z_1, z_2, \ldots z_{2n}$ with $\min\ell(z_i) \ge 0$.

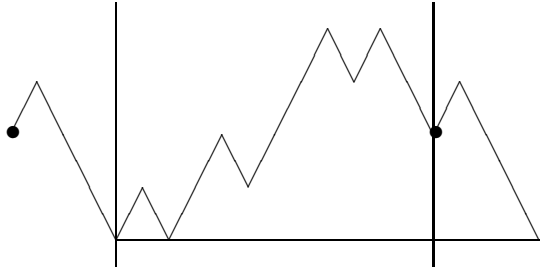
\begin{figure}[ht]
\begin{center}
\begin{picture}(270,120)
\put(28,58){$\bullet$}
\put(30,60){\line(1,2){10}}
\put(40,80){\line(1,-2){10}}
\put(50,60){\line(1,-2){10}}
\put(60,40){\line(1,-2){10}}
\put(70,20){\line(1,2){10}}
\put(70,10){\line(0,1){100}}
\put(80,40){\line(1,-2){10}}
\put(90,20){\line(1,2){10}}
\put(100,40){\line(1,2){10}}
\put(110,60){\line(1,-2){10}}
\put(120,40){\line(1,2){10}}
\put(130,60){\line(1,2){10}}
\put(140,80){\line(1,2){10}}
\put(150,100){\line(1,-2){10}}
\put(160,80){\line(1,2){10}}
\put(170,100){\line(1,-2){10}}
\put(180,80){\line(1,-2){10}}
\put(190,60){\line(1,2){10}}
\put(188,58){$\bullet$}
\put(190,10){\line(0,1){100}}
\put(200,80){\line(1,-2){10}}
\put(210,60){\line(1,-2){10}}
\put(220,40){\line(1,-2){10}}
\put(70,20){\line(1,0){160}}
\end{picture}
\caption{Picture of the path transformation. The original path connected the two black dots.
The front was cut off at the vertical line and added at the end.}
\end{center}
\label{fig:ell}
\end{figure}

\noindent
 Any good path created in this way can come from at most $2n$ arbitrary paths so the 
number of good paths is
 $$
\ge \frac{1}{2n} \binom{2n}{n} d^n
$$

\medskip
If we take two steps starting from a vertex with $a$ children then we return to a vertex with $a$ children.
We can go up twice in $ab$ ways, go up and then down in $a$ ways, go down and then up in $b$ ways,
 or go down twice in 1 way. Our first step is to count the paths of
length $2n$ that start from 0 and return to a vertex $x$ with $\ell(x)=0$. Using our result about pairs of steps, the number of paths is
the sum over $m$ of
$$
\pi_n(m) = \frac{n!}{m!(n-2m)!m!} 1^m (a+b)^{n-2m} (ab)^{m}
$$
Using Stirling's formula and dropping the $\sqrt{2\pi n}$ terms this is 
\begin{align*}
& \approx \frac{n^n}{m^{2m}(n-2m)^{n-2m}} (a+b)^{n-2m} (ab)^m \\
& = (m/n)^{-2m} (1- 2m/n)^{-(n-2m)} (a+b)^{n-2m} (ab)^m
\end{align*}
Taking the $n$th root of the last quantity, letting $x=m/n$, and taking logarithms
$$
(1/n) \log \pi_n(xn) = -2x\log x -(1-2x) \log(1-2x) + (1-2x) \log(a+b) + x \log(ab)
$$
To maximize this we take the derivative with respect to $x$ to get
\begin{align*}
& -2\log x -\frac{2x}{x} + 2\log(1-2x) - \frac{1-2x}{1-2x} \cdot (-2) - 2 \log(a+b) + \log(ab) \\
& = -2 \log(x) + 2\log(1-2x) -2 \log(a+b) + 2\log(\sqrt{ab})
\end{align*}
Setting the last quantity $=0$ we want $\sqrt{ab}(1-2x_0) = (a+b)x_0 $ or
$$
x_0 = \frac{\sqrt{ab}}{(a+b) + 2\sqrt{ab}} \qquad 1-2x_0 = \frac{a+b}{(a+b) + 2\sqrt{ab}}
$$
For this value of $x$
$$
(1/n)\log \pi_n(nx_0)  = -2x_0 \log(x_0/\sqrt{ab}) - (1-2x_0) \log( (1-2x)/(a+b))
$$
Writing $D=  \log( 1/(a+b + 2\sqrt{ab}))$ this is
$$
= -\frac{2\sqrt{ab}}{D} \log( 1/D) -  \frac{a+b}{D}  \log( 1/D)
$$
so we have 
$$
(1/2n)\log\pi_n(nx_0) = \frac{1}{2}\log(a+b + 2\sqrt{ab})
$$
Undoing the logarithm
$$
\pi_n(nx_0)^{1/2n} \approx (a+b+2\sqrt{ab})^{1/2} = \sqrt{a}+\sqrt{b}
$$
since $(\sqrt{a}+\sqrt{b})^2 = a + 2\sqrt{ab} + b$.

To finish up we note that $\sum_m \pi_n(m) \le n \pi_n(mx_0)$, i.e. the exponential order of magnitude of the sum is the same as the largest term. 
As in the homogeneous case we want to restrict our attention to paths that have $\ell(x_i) \ge 0$ to guarantee that they return to the root.
Let $\pi^0_n(m)$ be the number of such paths of length $2n$ with $m$ up steps.
We will use the same trick of permuting the steps. This time the steps are encoded by $2$, 0, and $-2$. However, a 0 step can consist of an up and a down. 
To assure this does not cause a problem, we suppose the first step in our path of length $2n$ is a $2$ and the last is a $-2$. Having done this we can safely 
permute the middle $2n-2$ steps to generate a path that stays $\ge 0$. We conclude that
$$
\pi^0_n(nx_0) \ge \pi_{n-1}(nx_0)/(n-1)
$$
and the proof is complete.

\clearp

\section{Proof of Theorem \ref{ablam2}} \label{sec:PfTh4}

To begin we recall the proof of $\lambda_2 \ge 1/2\sqrt{d}$ for the homogeneous tree.

\begin{proof} Let $\ell(x)$ be the ``height'' function on the tree defined in the introduction.
Let $\rho>0$ to be chosen later and let $W(A) = \sum_{x\in A} \rho^{\ell(x)}$. 
\begin{align*}
\frac{d}{dt} EW(\xi_t) & =  E\left( \sum_{x\in\xi_t} \left[ \lambda \sum_{y \not\in \xi_t: y\sim x} \rho^{\ell(y)} - \rho^{\ell(x)} \right] \right) \\
&\le E\left( \sum_{x\in\xi_t} \left[ \lambda \sum_{y : y\sim x} \rho^{\ell(y)} - \rho^{\ell(x)} \right] \right) \\
& = E\sum_{x\in \xi_t} [\lambda(d\rho + \rho^{-1}) - 1] \rho^{\ell(x)} = [\lambda(d\rho + \rho^{-1}) - 1] EW(\xi_t)
\end{align*}
If we take $\rho=1/\sqrt{d}$ then $d\rho+ \rho^{-1} = 2\sqrt{d}$ so if $\lambda< 1/2\sqrt{d}$ then $EW(\xi_t) \le -\ep EW(\xi_t)$ and the result
follows.
\end{proof}

To generalize the result suppose that we have a function $h \ge 0$ defined on the tree so that
\beq
\lambda[ d(x) h(x+) + h(x-) ] -h(x) = 0
\label{heq}
\eeq
where $x+$ is any of the vertices with $\ell(x+) = \ell(x)+1$, and $x-$ is the vertex with $\ell(x-)=\ell(x)-1$. 
To simplify the equation we suppose that $h(x+) = g(x+) h(x)$. 
In this case \eqref{heq} can be written as 
\begin{align*}
&\lambda\left[ d(x) \frac{h(x+)}{h(x)} + \frac{h(x-)}{h(x)} \right]  =  1 \\
&d(x) g(x+) + \frac{1}{g(x)} = 1/\lambda 
\end{align*}

\begin{proof}
If we let $g(a)$ and $g(b)$ be the $g$ values for vertices of degrees $m$ and $n$ then 
\begin{align*}
& a g(b) = \frac{1}{\lambda} - \frac{1}{g(a)} \\
& b g(a) = \frac{1}{\lambda} - \frac{1}{g(b)}
\end{align*}
To begin to solve the equation multiply the first equation by $bg(a)$ and subtract the second equation multiplied by $ag(b)$ to get
$$
\frac{bg(a) - ag(b)}{\lambda} - (b-a) = 0 \quad\hbox{or}\quad ag(b) = b g(a) + \lambda (a-b)
$$
Plugging this into the first equation gives
$$
b g(a) + (a-b) \lambda - 1/\lambda + 1/g(a) = 0 \quad\hbox{or}\quad  b g(a)^2 - ((b-a)\lambda + 1/\lambda) g(a) + 1 = 0
$$  
Using the formula for solutions of the quadratic equation the solutions are
$$
g(a) = \frac{ (b-a)\lambda + 1/\lambda \pm \sqrt{((b-a)\lambda + 1/\lambda)^2 - 4b}}{2b}
$$
In order for the root to be real we need $(b-a)\lambda + 1/\lambda > 2\sqrt{b}$. This = 0 when
$(b-a)\lambda^2 - 2\sqrt{b} \lambda + 1 = 0$, so
$$
\lambda = \frac{2\sqrt{b} \pm \sqrt{4b-4(b-a)}}{2(b-a)} = \frac{\sqrt{b} \pm \sqrt{a}}{b-a}
$$
The relevant solution for us is $\lambda = (\sqrt{b}-\sqrt{a})/(b-a) = 1/(\sqrt{b}+\sqrt{a}) \equiv \lambda_0$

Suppose $b>a$. $(b-a)\lambda + 1/\lambda \ge 0$ when $\lambda<\lambda_0$ so the roots for $g(a)$ are real both are positive. 
Note that when $\lambda = \lambda_0$
$$
g(a) = \frac{1}{2b} \left( \frac{(b-a)}{\sqrt{b}+\sqrt{a}} + \sqrt{b}+\sqrt{a} \right)  = \frac{1}{\sqrt{b}}
$$
By symmetry 
$$
g(b) = \frac{ (a-b)\lambda + 1/\lambda \pm \sqrt{((a-b)\lambda + 1/\lambda)^2 - 4a}}{2a}
$$
When $\lambda = \lambda_0$
$$
\frac{a-b}{\sqrt{b}+\sqrt{a}}{b-a} + \sqrt{a} + \sqrt{b}= 2\sqrt{a}
$$
so $g(b) = 1/\sqrt{a}$.
As $\lambda$ decreases $(a-b)\lambda$ become less negative while $1/\lambda$ becomes a larger positive number so $g(b)<0$.
Since the existence of positive solutions for $g(a)$ and $g(b)$ implies the process dies out we have proved Theorem \ref{ablam2}.
\end{proof}

To describe Pemantle's proof now, we note that to get a lower bound it is enough to have
$$
\lambda\left[ d(x) g(x+1) + \frac{1}{g(x)} \right] \le 1
$$
In the $a,b$ case this is 
\begin{align*}
a g(b) + \frac{1}{g(a)} \le \frac{1}{\lambda} \\
b g(a) + \frac{1}{g(b)} \le \frac{1}{\lambda}
\end{align*}
Now take $g(b) = 1/\sqrt{a}$ and $g(a) = 1/\sqrt{b}$ (the values we have when $\lambda=\lambda_0$) to prove
$$
\lambda_2 \ge \frac{1}{\sqrt{a}+\sqrt{b}}
$$ 

\clearp

\section{Proofs of Theorem \ref{abub2} and \ref{akub2}} \label{sec:PfTh5}

\subsection{Lower bound}

We denote the state of the contact process on the star graph by $j,k$ where $j$ is the number of infected leaves
and $k=1,0$ depending on whether the center is infected or not. Let $E_{j,k}$ be the expected vlaue for the Markov chain
starting from $j,k$. 
Let $T_{0,0}$ be the time until the state is $0,0$, i.e., infection dies out. Let $T_j$ be the first time there are $j$ infected leaves.

\begin{lemma}\label{ubsurv}
Let $K = \lambda n/(\lambda+1)$. For any $\epsilon>0$, the contact process on star graph has 
$$
E_{K,1} T_{0,0} \le (\log n) e^{(1+\epsilon)\lambda^2 n}
$$
when $n$ is sufficiently large.
\end{lemma}

\begin{proof}
The central vertex becomes healthy at rate 1. If this occurs when there are $k$ infected leaves then
the probability the process reaches $0,0$ before the center is infected is
$$
\left(\frac{1}{1+\lambda}\right)^k=e^{-k\log(1+\lambda)}\geq e^{-\lambda k}.
$$
When in $(0,(1+\epsilon)K]$, the rate at which a disaster occurs that takes the process to (0,0) is
$$
\geq e^{-(1+\epsilon)\lambda K}.
$$

We need to estimate the time between the end of one disaster and the start of another. Let $T$ denote the time spent when the center is healthy, i.e. the time a disaster takes.  The number of leaf infections $N$ that  will recover  while the center is healthy has a shifted geometric distribution with success probability $\lambda/(\lambda+1)$, i.e.,
$$
 P(N=j) = \left( \frac{1}{\lambda+1}\right)^{j} \cdot \frac{\lambda}{\lambda+1}
 \quad\hbox{for $j\ge 0$}.
$$

Then
$$E T\leq E (\sum_{i=0}^{N-1} \frac{1}{k-i})\leq \ln((1+\ep)K)+C \leq \ln K+C_1.$$

When the center is infected, each move takes time less than $ 1/\lambda n$ in expectation. Putting the two parts together, the expected time between two disasters is 
$$\leq \frac{1}{\lambda n}(\lambda(n-k)+k+1)+\ln K+C_1 \leq \ln K+ C_2.$$
Hence, the expected time spent in $[0,(1+\ep)K)$ before the process hits (0,0) is 
$$\leq e^{(1+\ep)\lambda K}\cdot ( \ln K+ C_2)=(\ln \lambda n +C_2)e^{(1+\ep)\lambda^2 n}.$$

To bound the excursions above $(1+\epsilon)K$,
we use the harmonic function for the birth and death chain in which the center of the star is always infected. It satisfies
$$
h_0(k) = \frac{\lambda(n-k)}{k + \lambda(n-k)} h_0(k+1) + \frac{k}{k + \lambda(n-k)} h_0(k-1),
$$
or rearranging 
$$
h_0(k+1)-h_0(k) = \frac{k}{\lambda(n-k)} [h_0(k)-h_0(k-1)].
$$ 
Since we only need this function for $k > K$ we can set $h_0(K)=0$, $h_0(K+1)=1$. In this case we will have
$$
h_0(K+j+1) - h_0(K+j) = \prod_{i=1}^j  \frac{K+i}{\lambda(n-K-i)}. 
$$
Since $K/\lambda (n-K)=1$
$$
\frac{K+i}{\lambda(n-K-i)} = \frac{K+i}{K} \frac{\lambda(n-K)}{\lambda(n-K-i)}
= \left( 1 + \frac{i}{K}\right) \left( 1 - \frac{i}{n-K} \right)^{-1}.
$$  
We will be interested in $j \ge K$ so we have
$$
h_0(K+j) \ge  h_0(K+j) - h_0(K+j-1)  \approx \exp( j(j-1)/2K ).
$$ 
From this it follows that
\begin{equation}\label{goup}
P_{K+1,1}(T_{(1+\epsilon)K}<T_K)\leq \exp( -\epsilon^2 K/2 ).
\end{equation}
At $K+1$ there is a probability $\geq (1/2\lambda n) e^{-(1+\epsilon)\lambda(K+1)}$ to hit 0,0 before the center becomes infected again, 
so the expected number of visits to $K+1$ before dying out is 
\begin{equation}\label{key}
\leq 2\lambda n e^{(1+\epsilon)\lambda(K+1)}.
\end{equation}
By \eqref{goup}, the expected number of visits to $(1+\epsilon)K$ is 
$$
\leq e^{-\epsilon^2K/2} \cdot 2\lambda n e^{(1+\epsilon)\lambda(K+1)}.
$$

Again we assume the center is always occupied and construct the process $Y_t$ which dominates the number of infected leaves.
\begin{align*}
Y_t \to Y_t +1 & \quad\text{ at rate $ \lambda(n-(1+\epsilon)K)$}\\
Y_t \to Y_t-1 & \quad\text{ at rate $(1+\epsilon)K$}
\end{align*} 
Let $T^{-}_m=\inf\{t: Y_t\leq m\}$. We want to estimate $E_{(1+\epsilon)K+1} T^-_{(1+\epsilon)K}$.

The drift of $Y_t$ is 
$$
\mu= \lambda(n - (1+\ep) K)) -(1+\epsilon)K = \lambda n - (1+\ep)\lambda n = -\ep \lambda n.
$$ 
Since $Y_t-\mu t$ is a martingale, the optional stopping theorem gives
$$
E_{(1+\epsilon)K+1,1} [Y(t\wedge T^-_{(1+\epsilon)K})-\mu(t\wedge T^-_{(1+\epsilon)K})] = (1+\epsilon)K+1,
$$
that is,
$$
E_{(1+\epsilon)K+1,1} [T^-_{(1+\epsilon)K}]= \frac{1}{-\mu}=\frac{1}{\epsilon\lambda n}.
$$
Therefore each excursion above $(1+\epsilon)K$ takes time $\leq 1/\epsilon\lambda n$ to get back to $(1+\epsilon)K$.
That is, the expected amount of time spent above $(1+\epsilon)K$ before dying out is 
$$
\leq e^{-\epsilon^2K/2} \cdot 2\lambda n e^{(1+\epsilon)\lambda(K+1)} \cdot \frac{1}{\epsilon\lambda n}\to 0 \quad\text{ as $n\to\infty$ }.
$$

Putting the two terms together, the expected survival time is 
$$\leq (\ln \lambda n +C_2)e^{(1+\ep)\lambda^2 n}+\frac{2}{\ep}e^{-\frac{\ep^2}{2}K+(1+\ep)\lambda (K+1)}\leq (\log n)e^{(1+\ep')\lambda^2 n}.$$ 
 
\end{proof}

\begin{lemma} \label{lbcr}
Given any $\epsilon>0$, the critical value $\lambda_2$ of the contact process on the $(1,n)$ tree satisfies
$$
\lambda_2\geq \sqrt{\frac{\frac{1}{2}(1-\epsilon)\log n}{n}}
$$
when $n$ is sufficiently large.
\end{lemma}

\begin{proof}
Let $x$ be a vertex with degree $n$ and $S_x$ be the star graph with center $x$.
Starting from the center $x$ infected, we run the contact process on $S_x$ until it does out. 
When a particle in $S_x$ gives birth onto a neighboring star $y$ we freeze the particle. These particles will be the descendants of $x$ in a branching random walk that we use to dominate the contact process on the $(1,n)$ tree. When the contact process on $S_x$ dies out, each frozen particle starts a new contact process on its star graph. If there are several frozen particles at the same site they start independent contact processes. Again we freeze every particle that escapes from $S_y$ , and so on. 

Let $B_{x,y}$ be the total number of infections reaching site $y$ for the contact process on the star $S_x$.
To upper bound $B_{x,y}$ we assume that the central vertex $x$ is always infected until the star dies out. Under this assumption 
the $B_{x,y}$ are independent. To bound the number of $B_{x,y}$ we replace the contact process by a branching random walk. 
If $\lambda = \sqrt{c (\log n)/n}$ then by Lemma \ref{ubsurv}, 
$$
EB_{x,y} \leq \lambda^2 e^{(1+\epsilon)\lambda^2 n} =\lambda^2 n^{(1+\epsilon)c}.
$$
Starting from the origin, there are 
$$
\le {2m \choose m}  {n^m} \cdot 1^m \le 2^{2m} n^m
$$ 
paths of length $2m$ that returns to it. So in expectation we have 
$$
\le 2^{2m} n^m (\lambda^2 n^{(1+\epsilon)c})^{2m}
\leq (4\lambda^4 n^{2(1+\epsilon)c+1})^m
$$
particles returning to the origin. When $c<\frac{1}{2}(1-\epsilon)$,
$$
4\lambda^4 n^{2(1+\epsilon)c+1}\leq \frac{\log^2 n}{n^2} \cdot n^{2-\epsilon^2}=\frac{\log^2n}{n^{\epsilon^2}}
$$
From this it follows that 
$$
\sum_{m=1}^{\infty} {2m \choose m} n^m (\lambda^2 n^{(1+\epsilon)c})^{2m}
\leq \sum_{m=1}^\infty \left( \frac{4\log^2n}{n^{\epsilon^2}} \right)^m <\infty.
$$
Therefore the expected number of infections that return to the origin is finite, which means the process does not survive locally.
\end{proof}

\clearp

\subsection{First upper bound}

In this section we will show that if $\ep>0$ and $n$ is large then 
$$
\lambda_2 \le \sqrt{(1+\ep)(\log n)/n}.
$$
To do this we will (i) get a lower bound on the time that a contact process on a star with $n$ leaves survives,
(ii) show that this is enough time so that with high probability the infection will be passed to an adjacent
vertex of degree $n$, and (iii) show that with high probability the new vertex experiences a long lasting infection. 

\begin{lemma} \label{super}
Let $\delta>0$. Suppose $\lambda = \sqrt{c(\log n)/n}$ and let 
$$
\theta=\frac{1}{\lambda+1}\left(  \lambda - \frac{1}{\delta\lambda n } \right)
$$ 
If $n$ is large then $h(X_t) \equiv (1-\theta)^{X_t}$ is a supermartingale when $X_t \le (1-4\delta)\lambda n$.
\end{lemma}

\begin{proof}
Suppose the current value is $V = (1-\theta)^{X_t}$ where $X_t \le (1-4\delta)\lambda n$.
New leaves become infected at rate $\lambda(n-k) \ge \lambda(n -\lambda n) \ge (1-\delta)n$ if $n$ is large, so
\begin{align*}
V \to V/(1-\theta) &\quad\hbox{at rate $k$} \\
V \to V(1-\theta) &\quad\hbox{at rate $\ge (1-\delta) \lambda n$} \\
V \to V(1-\theta)^{-Z} & \quad\hbox{at rate 1}
\end{align*}
The change in value due to the first two transitions is, if $\theta$ is small,
\begin{align*}
V \left(\frac{1}{1-\theta} - 1\right) \le (1-\delta)^{-1}\theta V &\quad\hbox{at rate $k$} \\
V[(1-\theta) -1] = -\theta V &\quad\hbox{at rate $\ge (1-\delta)\lambda n$} 
\end{align*}
When $k \le (1-4\delta)\lambda n <(1-\delta)(1-3\delta)\lambda n$ this is a net difference of $V\theta$ times
\beq
 [(1-\delta)^{-1}k - (1-\delta)\lambda n]  \le  -(2\delta \lambda n ) 
\label{part1}
\eeq
In the third case
\begin{align*}E(1-\theta)^{-Z} 
& = \sum_{k=0}^\infty \left(\frac{1}{1+\lambda}\right)^{k} \frac{\lambda}{1+\lambda} \cdot (1-\theta)^{-k} \\
& = \frac{\lambda}{1+\lambda} \sum_{k=0}^\infty  \left(\frac{1}{(1+\lambda)(1-\theta)}\right)^{k} \\
& = \frac{\lambda}{1+\lambda} \cdot \frac{1} { 1 - \frac{1}{(1+\lambda)(1-\theta)}}  \\
& = \frac{\lambda(1-\theta)}{\lambda - \theta - \theta\lambda}
\end{align*}
so we have 
$$
V(E(1-\theta)^{-Z} - 1) = V\theta \frac{1}{\lambda - \theta(1+\lambda)}
$$
For the chosen value of $\theta$ this is  $=\delta\lambda n$. Combining this with \eqref{part1}
so for any $\delta>0$ this is a supermartingale for large $n$.
\end{proof} 

Let $T_a^- = \min\{ t : X_t < a\}$ and let $T_b = \min\{ t : X_t=b\}$ and suppose $a < x < b$. Since $h(X_t)$ is a supermartingale
and $h$ is decreasing
$$
h(x) \ge h(a-1) P_x(T^-_a < T_b ) h(a-1) + h(b) [1 -  P_x(T^-_a < T_b ) ]
$$
Rearranging we have
$$
P_x(T^-_a < T_b ) \le \frac{h(x) -h(b)}{h(a-1)-h(b)}
$$
when $x=b-1$ this implies
\begin{align*}
P_x(T^-_a < T_b ) & \le \frac{h(b-1) -(1-\theta) h(b-1)}{h(a-1)-h(b-1)}\\
& = \frac{\theta}{h(a-1)/h(b-1) - 1} = \frac{\theta}{(1-\theta)^{a-b} - 1}
\end{align*} 

We will apply this result with $b = (1-4\delta)\lambda n$. Let $\eta>0$.
If $\delta$ is small $b \ge (1-\eta)\lambda n$. If $\lambda$ is small then
$1-\theta < 1-(1-\eta)\lambda$. Take $a=\eta L$. With these choices if $n$ is large.
\beq
P( T^-_a < T_b ) \le (1+\eta)\lambda \exp(-(1-\eta)^2\lambda^2 n)
\label{abub}
\eeq
If we replace $(1-\eta)^2$ by $1-2\eta$ the bound still holds. Let 
$$
G_L = \{ X_t\hbox{ returns }(1/2\lambda)e^{(1-3\eta)\lambda^2 n}\hbox{ times to $L$ before going }< \eta L \}.
$$
It follows from \eqref{abub} that
$$
P(G_L) \ge 1 - e^{-\eta\lambda^2 n}
$$
In order to return to $L$ we have to jump from $L-1$ to $L$, a time that dominates an exponential random variable with parameter $\lambda n/2$ so the law of large numbers tells us that the total amount of time before $X_t < \eta L$ is
\beq
\ge \frac{1}{\lambda^2 n} e^{(1-\eta)\lambda^2 n}
\label{timebd}
\eeq
with high probability (i.e., with a probability that tends to 1 as $n\to\infty$).

To prepare for the proof in the next section we will do the next few computations in general.
If $\lambda = \sqrt{(c+\ep) (\log n) /n }$ then this is
$$
\ge \frac{1}{2((c+\ep) \log n } n^{(c+\ep)(1-\eta_2)} \ge n^c
$$
for large $n$ if $\eta$ is small. Putting it all together, if we denote the state of the star by $(\ell,m)$ where $\ell$ is the number of infected leaves 

\begin{lemma} \label{survive}
Let $L = (1-4\delta)\lambda n$. and let $N_t$ be the number of infected leaves at time $t$. 
If $\delta$ and $\eta$ are chosen small enough then as $n \to \infty$
$$
P_{L,0}\left( \min_{t \le n^c}  N_t \ge \eta L \right) \to 1
$$
\end{lemma}

\mn
{\bf Push.} Suppose that the good event in Lemma \ref{survive} occurs. Up to time $n$ there are always at least $\eta L$ infected leaves. At the time we start to try to push the infection the center may not be occupied but the probability it does not become occupied by time 1 is $\le e^{-\eta L}$. With the center infected, the probability of passing the infection to the next degree $n$ vertex to the right is
$$
\ge \left( e^{-1} (1-e^{-\lambda}) \right)^2 \ge (\lambda/e)^2 \ge \frac{ C_1 \log n}{n} 
\quad\hbox{where $C_1=(1+\ep)e^{-2}$}
$$
Here we assume that each infection takes time $\le 1$. The lower bound comes from the fact that the pushing vertex has to survive for time 1, $e^{-1}$ and infect its neighbor by time 1. When $c=1$ we have $n/6$ chances to do this before time $n/2$ so the probability all fail is
\beq
\le \left( 1- \frac{C_1 \log n}{n} \right)^{n/6} \le \exp(-(C_1/6) \log n) \to 0.
\label{push}
\eeq

\mn
{\bf Ignite.} Now that we have the central vertex of the degree $n$ star infected, we need to increase the number of infected leaves to $L$ by time $n^2/2$.

\begin{lemma} \label{ignite}
Let $T_{0,0}$ be the first time the star is healthy. If $K = \lambda n/\sqrt{\log n}$, then for large $k$
\begin{align*}
& (i) P_{0,1}( T_K > T_{0,0} )  \le  3/\sqrt{\log n},  \\
& (ii) P_{K,1} (T_{0,0} < T_L)  \le  2\exp(- C \sqrt{\log n}) \\
& (iii) E_{0,1} \min\{ T_{0,0},  T_L \} \le 1
\end{align*}
\end{lemma}

\begin{proof}  
Let $p_0(t)$ be the probability a leaf is infected at time $t$ when there are no infected leaves at time 0 and
the central vertex has been infected for all $s \le t$. $p_0(0)=0$ and
$$
\frac{dp_0(t)}{dt} = - p_0(t) + \lambda(1-p_0(t)) = \lambda -(\lambda+1) p_0(t) 
$$
Solving gives 
$$
p_0(t) = \frac{\lambda}{\lambda + 1} ( 1 - e^{-(\lambda +1)t} )
$$
As $t \to 0$
$$
\frac{ 1 - e^{-(\lambda +1)t} }{ (\lambda +1)t } \to 1
$$
so if $t$ is small $p_0(t) \ge \lambda t/2$

Taking $t = 2/\sqrt{\log n}$ it follows that if $B =  \hbox{Binomial}(n,\lambda/\sqrt{\log n})$
$$
P_{0,1}( T_K < T_{0,0} ) \ge P( B > K) \exp(-2/\sqrt{\log n})
$$
The second factor is the probability that the center stays infected until time $1/\sqrt{\log n}$, and 
$$
\exp(-2/\sqrt{\log n})\ge 1 - 2/\sqrt{\log n}. 
$$
$B$ has mean $\lambda n/\sqrt{\log n}$ and variance $\le \lambda n/\sqrt{\log n}$ so Chebyshev's inequality implies
$$
P( B < \lambda n/(2\sqrt{\log n}) ) \le \frac{\lambda n/\sqrt{\log n}}{(\lambda n/(2\sqrt{\log n}))^2} \le \frac{4 \sqrt{\log n}}{\lambda n} \le n^{-1/2}
$$
For (ii) we use the supermartingale from Lemma \ref{super}. If $q=P_{K,1}(T_{0,0} < T_L)$
then we have
$$
q\le (1-\lambda/3)^{\lambda n/\sqrt{\log n}} \le \exp(- \lambda^2 n/3\sqrt{\log n} ) = \exp( - C \sqrt{\log n} )
$$

To bound the time we note that $EZ = (\lambda+1)/\lambda - 1 = 1/\lambda$ so
$$
\mu = \lambda n/2 - 1/\lambda
$$
gives a lower bound on the drift $X_t - \mu t$ is a submartingale
before time $V_L = T_{0,0} \wedge T_L$. Stopping this submartingale at the bounded stopping time $V_L \wedge s$ 
$$
EX(V_L\wedge s) - \mu E(V_L\wedge s) \ge EX_0=0.
$$
Since $EX(V_L\wedge s) \le L$, it follows that
$$
E(V_L\wedge s)\le L/\mu
$$ 
Letting $s\to\infty$ we have $EV_L\le L/\mu \le 1$ since $L=\lambda n/4$ and $1/\lambda \le \lambda n/2$.
\end{proof}

\mn
{\bf Comparison with oriented percolation.} We choose a bi-infinite path through the tree and look at the vertices of degree $n$ 
(which we call stars) along the path.
Pick some star on the path and number it 0. We label the stars on the path by $\ZZ$. Let ${\cal L} = \{ (k,\ell) \in 
Z^2 : k+\ell \hbox{ is even}\}.$ We say that $(k,\ell) \in {\cal L}$ is wet if the $k$th star has $\eta L$ infected leaves at time $\ell n$.

If this holds then Lemma \ref{ignite} implies that by time $\ell n + n/4$ the number of infected leaves will have reached $L$.
Lemma \ref{survive} then implies that with high probability the number of infected leaves will remain $\ge \eta L$ until time
$(\ell+1) n$. The estimate in \eqref{push} implies that before time $\ell n + 3n/4$ the centers of the stars at $k-1$ and $k+1$ will become infected.
Using Lemma \ref{ignite} again we see that with high probability $(k-1,\ell+1)$ and $(k+1,\ell+1)$ will both be wet. 
Call this event $G_{k,\ell}$. The events $G_{k,\ell}$ and $G_{k+2,\ell}$ are dependent but $G_{k,\ell}$ is independent of $G_{K+4,\ell}$ 
and of $G_{k',\ell'}$ when $\ell'\neq \ell$. 

In the terminology of \cite{StFlour} the percolation process is 1-dependent. Using Theorem 4.1 from that source,
we conclude  that there is positive probability that the process survives. To prove that it
survives locally we note that if $\eta^0_n$ and $\eta^1_n$ are oriented percolation starting with $\{0\}$ and all sites occupied
respectively and we let $r_n = \max\{ x : x \in \eta^0_n \}$ and $\ell_n = \min\{ x : x \in \eta^0_n \}$ then on $\eta_n^0 \neq \emptyset$
$$
\eta^0_n = \eta^1_n \cap [\ell_n,r_n]
$$
Theorem 4.2 in \cite{StFlour} implies that if the block event has probability close enough to 1 then
$$
\liminf_{n \to \infty} P( 0 \in \eta^1_{2n} ) \ge 19/20.
$$

\subsection{Closing the gap}

To bring the upper bound down to 
$$
\sqrt{(1/2+\ep)(\log n)/n}
$$ 
we need to take advantage of all of the stars, not just those on the path. 
When $c >1/2$ Lemma \ref{ignite} implies that if we have $\ge \eta L$ infected leaves at time 0 then with high probability
we have at least $L$ by time $n^c/4$. Lemma \ref{survive} then implies that with high probability 
by time $n^c/4$ and hence the number of infected leaves is always $\ge \eta L$ during $[n^c/4,3n^c/4]$. 

\mn
{\bf Step 1. Pushing the infection out to distance 2m.} 

\medskip
The first step in pushing to adjacent stars is for the leaves to infect the center 
within time one. When the number of infected leaves is always $\ge \eta L$,
the probability that this will take longer than one unit of time is $\le \exp(-\eta L)$. Using large deviations results for the Poisson process,
the probability there are more than $2\eta_1 L n^c$ arrivals in a rate $\eta_1L$ Poisson process run for time $n^c/2$ is 
$\le \exp(-\gamma L n^c)$ so if we let $G_1 = \{$ there is no interval of length $\ge 1$ between reinfections of the root $\}$, then $P(G_1) \to 1$ as $n \to\infty$.

The probability all attempts in a time interval of length $n^c/2$ fail to push the infection to a neighbor is 
$$
\left( 1- \frac{C \log n}{n} \right)^{n^{c}/6} \le \exp( - C  n^{c-1} \log n )
$$
so the probability of at least one success is 
\beq
\ge 1 - \exp( - C n^{c-1} ) \sim C n^{c-1}
\label{pushc}
\eeq 
as $n \to\infty$.

Once we push the infection to a new star, it has to survive there. Lemma \ref{ignite} implies that this probability tends to 1 as $n\to\infty$. In studying
the process on the path we could compare with 1-dependent oriented percolation. We do not have that option when we are working on the full tree
so we compare with bond-site percolation. The survival of an infection at a site and the lack of gaps of size 1 in which the center is not infected
are the site events. If we condition that the event $G_1$ occurs at a vertex then the pushing events for different neighbors are independent, 
so we can compare with an independent bond-site percolation process. We select one star to call the root. In the first phase of the
construction we only allow the infection to be passed to other stars that are children of the current vertex and are at distance 2.
We say that a star at distance $2m$ that is a descendant of the root is wet if it has $\ge \eta L$ infected leaves at time $mn^c$.
If the event $G_1$ occurs at a vertex, the number of neighboring stars that become infected is $\ge \hbox{binomial}(n,n^{c-1})$, which has
mean is $n^c$ and the variance is $\le n^c$. Let $Z_m$ be the number of wet stars at distance $2m$ at time $mn^c$, then $EZ_n = n^{c}$. 

\mn
{\bf Step 2. Bringing the infection back to the root.} 

\medskip
Each star at distance $2m$ from the root has a unique path back to the root.  By \eqref{pushc}
the probability for pushing the infection at each step is $\ge n^{c-1}$ so the mean number of infection paths $N_m$ that go out a distance 
$2m$ and lead back to the origin is
$$
EN_m \ge n^{(2c-1)m}.
$$ 
The different infection paths back to the root are not independent.
The number of paths from distance $2m$ back to the root that agree in the last $k$ steps is 
$$
\sim n^k (n^{m-k})^2.
$$
The probability that all the edges in the combined path are successful pushes is 
$$
n^{2(c-1)[k + 2(m-k)]}.
$$
Here and in the next step we are assuming that the probability of a successful push is exactly $n^{c-1}$
The expected number of successful pairs of paths outward to distance $2m$ that agree in the first $k$ steps is
$$
\sim  n^{(2c-1)[k + 2(m-k)]}.
$$
Thus the second moment of the number of successful paths out and back is
\begin{align*}
EN^2_m & \le \sum_{k=0}^m n^{(2c-1)[k + 2(m-k)]} \\
& \le n^{(2c-1)2m} \left( 1 + \sum_{\ell=1}^m n^{-(2c-1) \ell} \right).
\end{align*}
This implies that $(EN_m)^2/E(N_m^2) \to 1$. The Cauchy-Schwarz inequality implies that
$$
E\left( N_m 1_{\{N_m>0\}} \right)^2 \le E(N_m^2) P(N_m>0).
$$
Rearranging we conclude that 
$$
P(N_m>0) \ge (E N_m)^2 /E(N_m^2) \to 1.
$$ 
Since $m$ is arbitrary we have 
that the infection returns to the root at arbitrarily large times so we have established out lower bound on $\lambda_2$.

\subsection{ Extension to period k+1}

In this section we prove Theorem \ref{akub2}. As for the $(1,n)$ tree there are two things that must be proved.

\mn
{\bf Upper bound on $\lambda_c$.}
The structure of the proof is almost the same as the period 2 case,  so we content ourselves
to compute the constant. Suppose that the probability of pushing the infection across an edge is $n^{-a}$. The number of infections we
will generate per step as we work our way away from the root is $n^{b+1} n^{-a}$ and as we work our way back in is $n^{-a}$.
Thus to guarantee a lineage returning to the root at the end of the cycle we need $b+1 - 2a > 0$ or
\beq
a < \frac{b+1}{2}.
\label{aineq}
\eeq

If $\lambda = \sqrt{(c+\ep) (\log n)/n}$ then the expected number of times we will push the infection to a neighboring star is
$$
\frac{\lambda^{k+1}}{\lambda^2 n} \exp(\lambda^2 n) \ge n^{-1 -(k-1)/2 + c} = n^{-a},
$$ 
so we want $ c = (k+1)/2 - a$ or using \eqref{aineq} this is 
$$
c \ge \frac{k+1}{2} - \frac{b+1}{2},
$$
the constant given in the theorem.

\mn
{\bf Lower bound on $\lambda_c$.}
 
To prove the lower bound on the critical value, we follow the approach in the proof of Theorem \ref{lbcr}. Let $x$ be an infected star. We need to give an upper bound on the number of neighboring stars $y$ infected before it dies out. As before when $y$ is infected we freeze the particle and it will be a descendant in
the branching random walk that we use to upper bound the contact process on the stars in the graph. To estimate the number of times $y$ gets infected we 
work backwards from $y$ using a dual branching random walk. Paths are allowed only if the first step away from $y$ moves closer to $x$. The path is terminated when it achieves its goal of reaching $x$. If they reach $x$ while it is occupied, such paths will produce a particle at $y$

The shortest infection path from $y$ to $x$ has length $k+1$. By Lemma \ref{Etrails}, a path of length $\ell$ corresponds to $\lambda^\ell$ infection trails. 
If a path has length $k+1+2m$ then $m+k+1$ steps must go towards $x$ and $m$ away. There is only one way to go towards $x$ but $\le A = \max_{1\le i \le k} a_i$
ways to go away. If we encode a path as a sequence of $-$ for closer to $x$ and $+$ for further away then the number of strings of $m+k+1$ minuses and $m$ pluses is 
$\binom{2m+k+1}{m} \le 2^{2m+k+1}$. Thus the expected number of particles that start at $y$ and successfully make it to $x$ is
\begin{align*}
&\le \sum_{m=0}^\infty 2^{2m+k+1} A^{m} \lambda^{2m+k+1} \\
& \le 2^{k+1} \lambda^{k+1} \sum_{m=0}^\infty (2A \lambda^2)^{m} \le \frac{2^{k+1} \lambda^k}{ 1 - 2n^{-\ep}\log n } 
\end{align*}
since $A \le n^{1-\ep}$ and $\lambda = \sqrt{c(\log n)/n}$. This gives an upper bound of $C\lambda^{k+1}$ of pushing the infection
a distance $k+1$. 

Given this estimate the proof of the lower bound can be completed as before. Since the conclusion is different we redo the proof.
Let $B_{x,y}$ be the total number of infections reaching site $y$ for the contact process on the star $S_x$.
If $\lambda = \sqrt{c (\log n)/n}$ then by Lemma \ref{ubsurv}, 
$$
EB_{x,y} \leq \lambda^{k+1} e^{(1+\epsilon)\lambda^2 n} =\lambda^{k+1} n^{(1+\epsilon)c}.
$$
Let $N = n^{1+b}$. Starting from the origin, there are 
$$
\le {2m \choose m}  N^{m} \cdot 1^m \le 2^{2m} N^m
$$ 
paths of length $2m$ that returns to it. So in expectation we have 
$$
\le 2^{2m} N^m (\lambda^{k+1} n^{(1+\epsilon)c})^{2m}
\leq (2 \lambda^{k+1} n^{(1+b)/2} n^{(1+\epsilon)c})^{2m}
$$
particles returning to the origin. Let 
$$
c_0 = \frac{k+1}{2} - \frac{b+1}{2} 
$$
When $c<c_0(1-\epsilon)$, the sum is finite, so the expected number of infections that return to the origin is finite, which means the process does not survive locally.

\clearp

\section{Proofs for period three trees } \label{sec:per3}

\subsection{ Bound on $\lambda_g$.}

As stated in the introduction,
$$
\lambda_g = 1/ \hbox{ max eigenvalue of } \begin{pmatrix} 0 & a & 1 \\ 1 & 0 & b \\ c & 1 & 0 \end{pmatrix}
$$
Expanding about the first row 
\begin{align*}
 \hbox{det}  \begin{pmatrix} -x & a & 1 \\ 1 & -x & b \\ c & 1 & -x \end{pmatrix} & = -x(x^2-b) -a(-x-bc) +(1+cx) \\
&= -x^3 +(a+b+c) x +abc + 1
\end{align*} 
To solve $x^3-(a+b+c)x - (abc +1) = 0$ we use Cardano's formula to find a solution of  $x^3+px+q=0$
$$
\left( -\frac{q}{2} + \sqrt{ \frac{q^2}{4} + \frac{p^3}{27} } \right)^{1/3} 
+ \left( -\frac{q}{2} + \sqrt{ \frac{q^2}{4} - \frac{p^3}{27} } \right)^{1/3} 
$$
Plugging in $p=-(a+b+c)$ and $q=-(abc+1)$ this becomes
$$
\left( \frac{abc+1}{2} + \sqrt{ \frac{(abc+1)^2}{4} - \frac{(a+b+c)^3}{27} } \right)^{1/3} 
+\left( \frac{abc+1}{2} - \sqrt{ \frac{(abc+1)^2}{4} - \frac{(a+b+c)^3}{27} } \right)^{1/3} 
$$ 
which is not very informative.

To complete the proof we have to address the question: How do we know that the one root
produced by Cardano's formula is the right answer?
The discriminant of a cubic equation $ax^3+bx^2+cx+d=0$ is
\beq
\Delta = 18abcd - 4b^3d + b^2c^2 - 4ac^3 - 27a^2d^2
\label{discrim}
\eeq
When $\Delta>0$ there are three real roots, when $\Delta=0$ a double root, and when $\Delta<0$ two complex roots.
When $a=1$ and $b=0$ this becomes 
$$
-4c^3 -27d^2 = 4(a+b+c)^3 - 27(abc)^2
$$
When $a=2$, $b=3$, $c=4$, this is $4(9^3) - 27\cdot 24^2 < 0$ but this does not always hold.
When $a=b=1$ and $c=n$ this is $4(n+2)^3 - 27n^2$ which is positive for $n\ge 2$. 
Our formula is certainly is the right answer when the  the other two roots are complex. Since the largest eigenvalue is unique, 
it follows that this gives the maximum eigenvalue for all $a,b,c$.

\subsection{Bound on $\lambda_\ell$}

Based on the proof of Theorem \ref{ablam2} we want to solve
\begin{align*}
& a g(b) + 1/g(a) = 1/\lambda \\
& b g(c) + 1/g(b) = 1/\lambda \\
& c g(a) + 1/g(c) = 1/\lambda 
\end{align*}
Rearranging gives
\begin{align}
g(a) & = \frac{-1}{ag(b) - 1/\lambda} \label{ga}\\
g(b) & = \frac{-1}{bg(c) - 1/\lambda} \label{gb}\\
g(c) & = \frac{-1}{cg(a) - 1/\lambda} \label{gc}
\end{align}
Inserting \eqref{gb} into \eqref{ga}
\begin{align*}
&g(a)   = \frac{-1}{\frac{-a}{bg(c) - 1/\lambda} - 1/\lambda} \\
&-\frac{ag(a)}{bg(c) - 1/\lambda} - g(a)/\lambda = -1 \\
& -ag(a) - \left( \frac{g(a)}{\lambda} - 1 \right) \left( bg(c) - \frac{1}{\lambda} \right) = 0
\end{align*}
Now using \eqref{gc}
\begin{align*}
& -ag(a) - \left( \frac{g(a)}{\lambda} - 1 \right) \left( \frac{-b}{cg(a) - 1/\lambda}\right) 
+ \frac{1}{\lambda} \left( \frac{g(a)}{\lambda} - 1 \right) = 0 \\
& -ag(a)(cg(a) - 1/\lambda) + \left( \frac{g(a)}{\lambda} - 1 \right) b 
+ \frac{1}{\lambda} \left( \frac{g(a)}{\lambda} - 1 \right)\left( cg(a) - 1/\lambda \right) = 0
\end{align*}
Collecting terms
\begin{align*}
0 & = -ac g(a)^2 + \frac{(a+b)}{\lambda} g(a) - b\\
& + \frac{c}{\lambda^2} g(a)^2 - \left( \frac{c}{\lambda} + \frac{1}{\lambda^3}\right) g(a) + \frac{1}{\lambda^2} \\
& = \alpha g(a)^2 + \beta g(a) + \gamma
\end{align*} 
where 
\beq
\alpha =  -ac + \frac{c}{\lambda^2} \qquad \beta = \frac{a+b-c}{\lambda} - \frac{1}{\lambda^3}
\qquad \gamma = -b + \frac{1}{\lambda^2}
\label{ga}
\eeq
Let $d=a+b-c$. There is a solution if 
$$
0 \le \beta^2 - 4\alpha\gamma  = \frac{d^2}{\lambda^2} - 2\frac{d}{\lambda^4} + \frac{1}{\lambda^6}
- 4 \left(abc - \frac{bc+ac}{\lambda^2} + \frac{c}{\lambda^4} \right)
$$
A little algebra shows
$$
d^2 = a^2+b^2+c^2 -2ac -2bc +2ab \quad \hbox{so we have} \quad d^2 + 4ac+4bc = (a+b+c)^2
$$
and $-2d - 4c = -2(a+b+c)$ so 
$$
0 \le  - 4abc + \frac{(a+b+c)^2}{\lambda^2} - \frac{2(a+b+c)}{\lambda^4} + \frac{1}{\lambda^6}
$$
Changing variables 
$$
D(x) \equiv  x^3 - 2(a+b+c)x^2 + (a+b+c)^2 x - 4abc
$$

\begin{lemma}
The cubic equation $D(x)=0$   has three real roots.
\end{lemma}

\begin{proof}
The cubic discriminant of $D$, $\Delta$ defined in \eqref{discrim}, is equal to 
\begin{align*}
\Delta & = 18 \cdot 8(a+b+c)^3 abc - 4 \cdot 8(a+b+c)^3 \cdot 4abc \\
& + 4(a+b+c)^2 \cdot (a+b+c)^4 - 4(a+b+c)^6 - 27 \cdot (4abc)^2 \\
& = 16abc (a+b+c)^3 -27(4abc)^2 \\
& = 16abc(a^3 + b^3 + c^3 + 3a^2b + 3a^2c + 3b^2a + 3b^2c +3c^2 a + 3c^2 b - 21 abc)
\end{align*}
The arithmetic-geometric mean inequality implies
\beq
a^3 + b^3 + c^3 \ge 3abc
\label{AMGM}
\eeq
The rearrangement inequality says that if $x_1 \le x_2 \ldots x_n$ and  $y_1 \le y_2 \ldots y_n$ then
for any permutation $\sigma$
$$
x_ny_1 + \cdots + x_1y_n \le x_{\sigma(1)}y_1 + \cdots + x_{\sigma(n)} y_n \le x_1y_1 + \cdots + x_ny_n
$$

Applying the rearrangement inequality to $a,b,c$ and $ab,ac,bc$ gives
\begin{align*}
3a^2b + 3ac^2 + 3b^2c  \ge 9abc & \quad\hbox{permutation } a,c,b  \\
3ab^2+ 3a^2c +3bc^2  \ge 9abc & \quad\hbox{permutation } b,a,c
\end{align*} 
Using these inequalities with \eqref{AMGM}, which is strict when $a,b,c$ are not all equal, we have
$\Delta>0$ which proves the desired conclusion.
\end{proof}

\begin{lemma}
All three roots $r_1, r_2, r_3$ are positive.
\end{lemma}

\begin{proof} Since $D(x) = (x-r_1)(x-r_2)(x-r_3)$, we have
\begin{align*}
r_1+r_2+r_3 &= 2(a+b+c) > 0 \\
r_1r_2 + r_2r_3 + r_1r_3 &= (a+b+c)^2 > 0 \\
r_1r_2r_3 & = 4abc >0
\end{align*}
The third equation implies that either all of the roots are positive or only one is. To rule out the second case
suppose, without loss of generality, that $r_2,r_3<0$. $r_1+r_2+r_3>0$ implies 
$$
r_1 > -(r_2+r_3) 
$$
and hence $ r_1(r_2+r_3) < -(r_2+r_3)$. Adding $r_2r_3$ and using the second inequality we have
$$
0 < r_1(r_2+r_3) + r_2r_3 < -(r_2+r_3)^2 + r_2r_3 = -r_2^2 - r_2r_3 - r_3^2 < 0
$$
a contradiction.
\end{proof}

Let $x_0$ be the largest root and note that
$$
D(a+b+c) \equiv  - 4abc + (a+b+c)^3 c - 2(a+b+c)^3 + (a+b+c)^3 = -4abc
$$
so $x_0 \ge (a+b+c)$. We want to show that 

\begin{lemma}
If $x > x_0$ then  $g(a),g(b),g(c)>0$. 
\end{lemma}

\begin{proof}
When a quadratic
$\alpha x^2 + \beta x + \gamma = 0$ has $\alpha,\gamma >0$ and $\beta <0$ it has two positive roots.
Using \eqref{ga} and then symmetry we want
$$
\begin{matrix}  
-ac + \frac{c}{\lambda^2} >0 & \frac{a+b-c}{\lambda} - \frac{1}{\lambda^3} < 0 & -b + \frac{1}{\lambda^2} > 0 \\
-ba + \frac{a}{\lambda^2} >0 & \frac{b+c-a}{\lambda} - \frac{1}{\lambda^3} < 0 & -c + \frac{1}{\lambda^2} > 0 \\
-cb + \frac{b}{\lambda^2} >0 & \frac{a+c-b}{\lambda} - \frac{1}{\lambda^3} < 0 & -a + \frac{1}{\lambda^2} > 0
\end{matrix}
$$
if $x>x_0$ then $1/\lambda^2 = x > a+b+c$ by the computation before the lemma. Using this it is easy to see that all the desired inequalities hold.
\end{proof}

The next step is to solve the cubic $D(x)=0$. It is probably easiest to do this numerically, which is what we have done in our examples,
but for  completeness we describe the process to get
an exact solution. To apply Cardano's formula we have to eliminate the $x^2$ term. To do this we let $x = y + 2\sigma/3$
where we have introduced $\sigma = a+b+c$ to shorten the formula
\begin{align*}
x^3 & = y^3 + 3 (2\sigma/3) y^2 + 3 (2\sigma/3)^2 y + (2\sigma/3)^3 \\
-2\sigma x^2 &= -2\sigma y^2 - 2\sigma \cdot 2 (2\sigma/3) y  - 2\sigma (2\sigma/3)^2 \\
\sigma^2 x & = \sigma^2 y + \sigma^2 \cdot (2\sigma/3)
\end{align*}
so the new polynomial is
\begin{align*}
0 &=  y^3 + \sigma^2 y [ 4/3 - 8/3 + 1] + \sigma^3 [8/27 - 8/9 + 2/3] - 4abc \\
& = y^3 - \frac{\sigma^2}{3} y + \frac{2\sigma^3}{27} - 4abc
\end{align*}
This can be used to solve the cubic for $y$. After that we set $x=y+2\sigma/3$. Again there is the issue of whether
Cardano's formula gives the solution we want. Since there are always three real roots this will be true if it holds
at one point. However, we have not investigated this question.

\section*{Acknowledgments}
The first four authors are undergraduates who were supported by 
the DOMath program at Duke for eight weeks in the summer of 2018.
XH was a graduate student assistant working with the project. RD was partially supported by NSF grant DMS 1505215 from the probability program.

\end{document}